\documentclass[10pt,a4paper]{article}

\usepackage{psfig}
\usepackage{graphics,times,mathptm}
\usepackage{latexsym,amssymb}
\usepackage{epsfig}
\newenvironment{keywords}{ \noindent {\small\bf Key Words}:}{ }
\newenvironment{proof}{ \noindent {\bf Proof}.}{ }

\newtheorem{lemma}{Lemma}
\newtheorem{theorem}{Theorem}

\def\beq{\begin{equation}}
\def\eeq{\end{equation}}
\def\bd{\begin{description}}
\def\ed{\end{description}}
\def\bea{\begin{eqnarray}}
\def\eea{\end{eqnarray}}
\def\beas{\begin{eqnarray*}}
\def\eeas{\end{eqnarray*}}

\newcommand{\ds}{\displaystyle}
\newcommand{\nms}{\normalsize}

\begin{document}

\author{ Yaroslav D. Sergeyev}

\title{Univariate global optimization with multiextremal
non-differentiable constraints  without  penalty
functions\thanks{This research was  supported by the
       following grants: FIRB RBNE01WBBB, FIRB RBAU01JYPN, and RFBR
       04-01-00455-a. The author thanks Prof. D.~Grimaldi for proposing the application discussed
       in the paper.}}

\author{{\small \bf Yaroslav D. Sergeyev}\\ \\ [-2pt]
         \nms D.E.I.S. -- Universit\`a della Calabria, 87036 Rende (CS)
         -- Italy\\[-4pt]
         \nms and University of Nizhni Novgorod,\\[-4pt]
         \nms Gagarin Av., 23, Nizhni Novgorod -- Russia\\[-4pt]
         \nms {\tt (yaro@si.deis.unical.it)}
}

\date{}

\maketitle

%

%

\begin{abstract}
This paper proposes a new algorithm for solving  constrained
global optimization problems where both the objective function and
constraints  are one-dimensional non-differentiable multiextremal
Lipschitz functions. Multiextremal constraints can lead to complex
feasible regions being collections of isolated points and
intervals having  positive lengths. The case is considered where
the order the constraints are evaluated is fixed by the nature of
the problem and a constraint $i$ is defined only over the set
where the constraint $i-1$ is satisfied. The objective function is
defined only over the set where all the constraints are
satisfied. In contrast to traditional approaches, the new
algorithm does not use any additional parameter or variable. All
the constraints are not evaluated during every iteration of the
algorithm providing a significant acceleration of the search. The
new algorithm either finds lower and upper bounds for the global
optimum or establishes that the problem is infeasible.
Convergence properties and numerical experiments showing a nice
performance of the new method in comparison with the penalty
approach are given.
\end{abstract}

\hspace{2cm}


\newpage

\section{Introduction}
\label{intro}
 In last decades univariate global optimization
problems were studied intensively (see
\cite{Calvin_and_Zilinskas_(1999),Hansen_(3),Lamar_(1999),Locatelli_and_Schoen_(1995),
MacLagan_and_Sturge_and_Baritompa_(1996),Pijavskii_(1972),Sergeyev_(1998),Strongin_(1978),Wang_and_Chang_(1996)})
because there exists a large number of real-life applications
where it is necessary to solve such problems (see
\cite{Brooks_(1958),Hansen_(3),Patwardhan_(1987),Ralston_(1985),Sergeyev_et(1999),Strongin_(1978),Strongin_and_Sergeyev}).
On the other hand, it is important to study these problems because
mathematical approaches developed to solve them can be generalized
to the multidimensional case by numerous  schemes (see, for
example, one-point based, diagonal, simplicial, space-filling
curves, and other popular approaches in
\cite{Floudas_and_Pardalos_(1996),Horst_and_Pardalos_(1995),Horst_and_Tuy_(1996),Mladineo_(1992),Pardalos_and_Rosen_(1990),Pinter_(1996),Strongin_(1978)}).

Electrotechnics and electronics are among the fields where
one-dimensional global optimization methods can be used
successfully (see
\cite{Casado_GS(2000),Casado_GS(2002),DaponteGMS95,DaponteGMS96,Molinaro_(2001),Sergeyev_et(1999),Strongin_and_Sergeyev}).
Let us consider, for example, the following so-called `mask
problem' for transmitters. We have a transmitter (for instance,
that of GSM cellular phones) that in a frequency interval $[a,b]$
has an amplitude $A(x)$ that should be inside the mask defined by
functions  $l(x)$ and $u(x)$, i.e., it should be  $l(x) \le A(x)
\le u(x)$. The mask is defined by international rules agreed to
avoid interference appearing when amplitude is too high for a
given frequency  and by properties of electronic components used
to construct the transmitter. Then, it is necessary to find a
frequency $x^* \in [a,b]$ such that the power, $p(x)$, of the
transmitted signal is maximal.

It happens often in engineering optimization problems (see
\cite{Strongin_(1978),Strongin_and_Sergeyev}) that if a
constrained is not satisfied at a point then many other
constraints and the objective function are not defined at this
point. This situation holds in our mask problem because if for a
frequency $\xi$ if happens that $A(\xi) > u(\xi)$ or $A(\xi) <
l(\xi)$ then there is no transmission and the function $p(\xi)$ is
not defined. Since the amplitude can touch the mask both from its
internal and its external parts, isolated points in the admissible
region of $p(x)$ can take place. If the maximal power $p(x^*)$ is
attained at an isolated point $x^*$, then this point should be
discarded from consideration because it cannot be realized in
practice. Thus, the solution is acceptable only if it belongs to a
finite interval of a certain length.

This problem can be reformulated in the following   general
framework of   global optimization problems   considered in this
paper. It is necessary to find  the global minimizers and the
global minimum of a function $f(x)$ subject to constraints
$g_{j}(x) \le 0, 1\le i \le m,$ over an interval $[a,b]$. The
objective function $f(x)$ and constraints $g_{j}(x), 1\le i< m,$
are multiextremal non-differentiable `black-box' Lipschitz
functions with a priori known Lipschitz constants (to unify the
description process  the designation $g_{m+1}(x) \triangleq f(x)$
is used hereinafter). Very often in real-life applications the
order  the constraints are evaluated is fixed by the nature of the
problem and not all the constraints are defined over the whole
search region $[a,b]$. The worst case  is considered here, i.e., a
constraint $g_{j+1}(x)$ is defined only over subregions where
$g_{j}(x)\le 0$. This means that if a constraint is not satisfied
at a point, the rest of constraints and the objective function are
not defined at that point.  The sets $Q_{j}, 1\le j\le m+1,$ can
be so defined as follows
\begin{equation}
Q_{1}=[a,b], \hspace{3mm} Q_{j+1}=\{x\in Q_{j}: g_{j}(x)\le 0 \},
\hspace{5mm} 1\le j\le m,  \label{Q}
\end{equation}
\[
Q_{1}\supseteq  Q_{2}\supseteq\ldots\supseteq Q_{m}\supseteq
Q_{m+1}.
\]
Since the constraints are multiextremal, the admissible region
$Q_{m+1}$ and regions $Q_{j},$ $1\le j\le m,$ can be   collections
of intervals having positive lengths and isolated points.
Particularly, isolated points appear when one of the constraints
touches zero, for example, if $g_{j}(x)$ is the square of some
function, then $g_{j}(x) \le 0$ only when $g_{j}(x)=0$. To be
implementable in practice, optimal solutions should have a
feasible neighborhood of positive length thus, an additional
constraint is included in the model: a point $x^*$ should belong
to an admissible interval having length equal to or greater than
$\delta > 0$ -- a value supplied by the final user. The set of all
such intervals is designated as $Q^{\delta}$ (of course,
$Q^{\delta} \subseteq Q_{m+1}$). Eventually found isolated points
and feasible subregions having length less than $\delta$ should be
excluded from consideration. If the case of infeasible problem
$Q^{\delta}=\emptyset$ holds, it should be also determined.

We can now state the problem formally. Find  the global minimizers
$x^*$ and the corresponding value $f^*$ such that
\begin{equation}
f^*=f(x^*)=\min \{ f(x): x\in Q^{\delta} \},    \label{problem}
\end{equation}
where the objective function $f(x)$ and constraints $g_{j}(x),
1\le i< m,$ are multiextremal  functions  satisfying the Lipschitz
condition in the form
\begin{equation}
\mid g_{j}(x') - g_{j}(x'')\mid  \le  L_{j}\mid x'- x''\mid ,
\hspace{3mm} x',x''\in Q_{j},   \hspace{3mm} 1\le j\le m+1,
\label{Lip}
\end{equation}
and the constants
\begin{equation}
0 <  L_{j} < \infty, \hspace{5mm} 1\le j\le m+1 \, , \label{const}
\end{equation}
are known (this  supposition is classical in global optimization
(see \cite{Hansen_(3),Horst_and_Tuy_(1996),Pijavskii_(1972)})).
Methods working on the basis of this assumption are called `exact'
in literature, methods estimating these values are `practical'. On
the one hand, the exact methods serve as a basis for studying
theoretical properties of practical ones and are used as a unit of
measure of the speed of practical methods. On the other hand, in
certain cases, when additional information about the objective
function and constraints is available, they can be applied
directly.

An example of such a problem is shown in Fig.~\ref{f1}.  It has
two non-differentiable multiextremal constraints $g_{1}(x)$ and
$g_{2}(x)$. The corresponding sets $Q_{1}=[a,b], Q_{2},$  and
$Q_{3}$ are shown. The point $c$ belongs to the sets $Q_{1},
Q_{2},$  and $Q_{3}$ but $c \notin Q^{\delta}$.  The set
$Q^{\delta}$ is shown by the grey color. It can be seen from
Fig.~\ref{f1} that the sets $ Q_{2}, Q_{3},$ and $ Q^{\delta}$
consist of disjoint subregions and $ Q_{2}, Q_{3}$ contain also
an isolated point.
\begin{figure}[t]
  \begin{center}
    \caption{An example of the
problem (\ref{problem})--(\ref{const}).  \label{f1}}
    \epsfig{ figure = 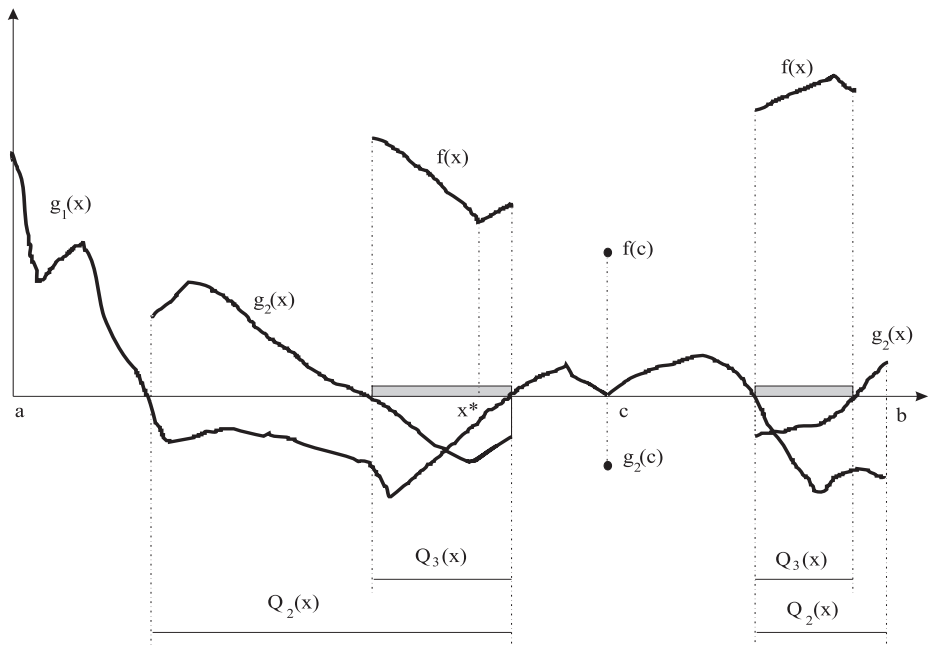, width = 4.8in, height = 3.375in,  silent = yes }
  \end{center}
\end{figure}

It is not easy to find a traditional algorithm for solving the
problem (\ref{problem})--(\ref{const}). For example, the penalty
approach  requires that $f(x)$ and $g_i(x),\, 1\le i\le m,$ are
defined over the whole search interval $[a,b]$. It seems that
missing values can be simply filled in with either a big number
or the function value at the nearest feasible point.
Unfortunately, in the context of Lipschitz algorithms,
incorporating such ideas can lead to infinitely high Lipschitz
constants, causing degeneration of the methods and
non-applicability of the penalty approach.

A promising approach called the {\it index scheme} has been
proposed in \cite{Strongin_(1984)} (see also
\cite{Strongin_and_Markin_(1986),Strongin_and_Sergeyev}) in
combination with stochastic Bayesian algorithms. An important
advantage of the index scheme is that it does not introduce
additional variables and/or parameters by opposition to classical
approaches in
\cite{Bertsekas_(1996),Bertsekas_(1999),Horst_and_Pardalos_(1995),Horst_and_Tuy_(1996),Nocedal_and_Wright_(1999)}.
It has been recently shown in \cite{Sergeyev_Famularo_Pugliese}
that the index scheme can be also successfully used in combination
with the Branch-and-Bound approach. Unfortunately, this scheme can
not be  applied directly for solving the problem
(\ref{problem})--(\ref{const}) because it has good convergence
properties when all the sets $Q_{j}, 1\le j\le m+1,$ have no
isolated points -- requirement hardly verified in practice without
some additional information about the problem.

Thus, isolated points give serious problems when one has only
Lipschitz information. First, because it is not possible to say a
priori whether the feasible region has isolated points or not (for
example, the method from \cite{Sergeyev_Famularo_Pugliese}
converges only to global minimizers if it is ensured absence of
isolated points). Second, in Lipschitz global optimization
isolated points can lead to two problems: accumulation of trial
points in their neighborhood (this happens even if there exists a
sub-region where a constraint does not touch zero but is only
close to zero) and increasing the estimates of Lipschitz constants
to infinity. This fact means that the search region will be
covered by a uniform mesh of trials) if the Lipschitz constant is
estimated or the method simply will not work  if the Lipschitz
constant is given -- our case -- because local adaptively obtained
information will contradict the given one. Therefore, traditional
Lipschitz methods cannot be used in the presence of isolated
points and, since their absence can be hardly determined in
practice, developments of methods that are able to work
independently of the presence or absence of them becomes very
important.

In this paper, such a method is proposed.  It evolves  the idea of
separate consideration of each constraint introduced in
\cite{Strongin_(1984)} in a new way   and reduces the original
constrained problem to a new continuous problem. The method from
\cite{Sergeyev_Famularo_Pugliese}   is used as a basis for
construction of the new scheme. Instead of discontinuous support
functions   proposed in \cite{Sergeyev_Famularo_Pugliese}, new
continuous functions are built. These new structures are very
important because by using them it becomes possible to apply
numerous tools developed in Lipschitz unconstrained optimization
to a very general class of  constrained problems. It is also
necessary to emphasize that the new approach does not introduce
additional variables and/or parameters during this passage from
initial discontinuous constrained partially defined problem to the
continuous unconstrained one.

To conclude this introduction it is necessary to emphasize once
again that the problem of multi-dimensional extensions of
one-dimensional Lipschitz global optimization methods to many
dimensions is a non-trivial serious problem (P.~Hansen and
B.~Jaumard write in their survey on Lipschitz optimization
\cite{Hansen_(3)} published in the Handbook of Global
Optimization: `Large problems (with 10 variables or more) appear
to be often intractable, at least if high precision is required')
and is beyond the scope of this paper dedicated to the univariate
algorithms and univariate applications. However, the approach
proposed here is very promising from this point of view. In the
future,  a number of various multi-dimensional extensions
(starting from the adaptive diagonal and space-filling curves
approaches (see \cite{Sergeyev_(2000),Strongin_and_Sergeyev}))  of
the algorithm presented in this paper will be studied.

The rest of the paper is organized as follows. The new method is
described in Section~2. Section~3 contains computational results
and a brief conclusion.

\section{Continuous index  functions and the new algorithm}

 The  index scheme  (see
\cite{Strongin_(1984),Strongin_and_Markin_(1986),Strongin_and_Sergeyev})
considers constraints one at a time at every point where it has
been decided to calculate $f(x)$ determining the  index $\nu =
\nu(x),   1 \le \nu \le m+1,$   by the following conditions
\begin{equation}
g_{j}(x)\le 0,  \hspace{3mm}1\le j\le \nu -1, \hspace{5mm} g_{\nu
}(x)>0,       \label{index}
\end{equation}
where for $\nu =m+1$ the last inequality is omitted. The term
{\it trial} used hereinafter means determining the index $\nu
(x)$  at a point $x$ by evaluation $g_{i}(x), 1\le i\le \nu
(x)$.  The index $\nu (x)$ and the value $g_{\nu (x)}(x)$ are
called {\it results of the trial}.

The discontinuous index function $J(x), x \in [a,b],$ can be
written for the problem (\ref{problem})--(\ref{const}) following
\cite{Strongin_(1984)}
\begin{equation}
J(x) = g_{\nu(x) }(x) - \left\{ \begin{array}{ll}
                             0 \, ,          & \nu (x)<m+1, \\
                             f^{*} \, ,& \nu (x)=m+1,
                                    \end{array}
                            \right.          \label{indexFun1}
\end{equation}
where the value $f^{*}$ is the unknown solution to this problem.

Let us start our theoretical consideration by noticing that the
global minimizer of the original constrained problem
(\ref{problem})--(\ref{const})  in the case $Q^{\delta}\neq
\emptyset$ coincides with the so\-lu\-tion to the following
unconstrained discontinuous problem
\begin{equation}
 J(x^{*})=\min \{ J(x): x\in \overline{Q}^{\delta}\},  \label{7}
 \end{equation}
where
\begin{equation}
\overline{Q}^{\delta}= [a,b]\setminus \{ Q_{m+1} \setminus
Q^{\delta}\}.                              \label{compQ}
\end{equation}

Suppose now that trials have been executed in a way at some points
\begin{equation}
a=x_{0}<x_{1}<\ldots  <x_{i}<\ldots  <x_{k}=b \label{indexFun4}
\end{equation}
and $\nu_{i}=\nu(x_{i}), 0\le i\le k,$ are their {\it starting
indexes}.  Note that the notion of {\it index} is different with
respect to
\cite{Strongin_(1984),Strongin_and_Markin_(1986),Strongin_and_Sergeyev}
where the index is calculated once and then used in the course of
optimization. In this paper, formula (\ref{index}) defines the
starting value for the index that can then be changed during the
work of the algorithm.

The  points from (\ref{indexFun4}) form the list (called
hereinafter {\it History List $H(k)$}) of intervals
$[l_{i},r_{i}], 1\le i\le k,$ where
\[
l_{i} < r_{i} ,  \hspace{5mm} 1\le i\le k,
 \]
 \[
r_{i} = l_{i+1}, \hspace{5mm} 1\le i < k.
 \]
The record $x \in H(k)$ means that the point $x=l_i$ or $x=r_i$
for an interval $i$ from $H(k)$. Every element $i, 1\le i\le k,$
of the list contains the following information:
\begin{equation}
[l_{i},r_{i}],\hspace{2mm} \nu(l_{i}),\hspace{2mm}
\nu(r_{i}),\hspace{2mm} g_{\nu (l_{i})}(l_{i}),\hspace{2mm} g_{\nu
(r_{i})}(r_{i}).           \label{infoHk}
 \end{equation}
The second list, $W(k)$,  called {\it Working List } is built
during the work of the method to be introduced by excluding from
$H(k)$ intervals where global minimizers of the problem
(\ref{problem})--(\ref{const}) can not be located (initially it is
stated $W(k)=H(k)$). In contrast to $H(k)$ where the information
(\ref{infoHk}) is calculated once and then is kept during the
search,   indexes $\nu(l_{i})$ and $ \nu(r_{i})$ in $W(k)$ can be
changed in the course of optimization.

In order to pass from the problem (\ref{problem})--(\ref{const})
to  the problem (\ref{7}) it is necessary to estimate the value
$f^{*}$ from (\ref{problem}) and the set $Q^{\delta}$. Using the
results of trials at the points from the row (\ref{indexFun4})
the value
 \beq
  Z^{*}_{k}=\min \{ g_{m+1}(x): \nu (x)=m+1, x \in W(k)  \}. \label{Zk}
 \eeq
estimating $f^{*}$ can be calculated if there exist points $x$
with the index $\nu (x)=m+1$. This value allows us to define the
function $ J^k(x), x \in [a,b],$ by replacing the unknown value
$f^{*}$ in (\ref{indexFun1}) by $Z^{*}_{k}$:
\begin{equation}
J^k(x) = g_{\nu(x) }(x) - \left\{ \begin{array}{ll}
                             0 \, ,          & \nu (x)<m+1, \\
                             Z^{*}_{k} \, ,& \nu (x)=m+1.
                                    \end{array}
                            \right.          \label{Jk}
\end{equation}
The following Lemma establishes some useful properties of the
functions $J(x)$ and $J^k (x)$.

\begin{lemma}
 \label{l1}
The following assertions hold for the functions $J(x)$ and $J^k
(x)$:
\begin{enumerate}
\item[i.]
 for all points $x$ having indexes $\nu(x)<m+1$,
it follows  $J^k (x)=J(x)>0$;
\item[ii.]
\beq
  J^k (x) \le 0, \hspace{5mm} x \in \{ x: g_{m+1}(x) \le
  Z^{*}_{k} \}.       \label{l1.1}
 \eeq
\item[iii.]
if   $\nu(x)=m+1$ and $Z^{*}_{k} \ge f^*$ then
 \beq
J^k (x) \le  J(x).  \label{l1.2}
 \eeq
\end{enumerate}
\end{lemma}

\begin{proof}
Truth of assertions i -- iii follows from definitions of   the
functions $J(x)$ and $J^k (x)$.   \rule{5pt}{5pt}
 \end{proof}

Particularly, it follows from Lemma~\ref{l1} that (\ref{l1.2})
holds if the trial point corresponding to $Z^{*}_{k}$ belongs to
$Q^{\delta}$. The estimate (\ref{l1.2}) is not   true if $
Z^{*}_{k} \le f^*$, the situation which can occur only if
$x^{*}_{k}$ belongs to the set $Q_{m+1} \setminus Q^{\delta}$
where
 \beq
 x^{*}_{k}=\arg \min \{ g_{m+1}(x): \nu (x)=m+1, x \in W(k)  \}. \label{xk}
 \eeq

Let us introduce the following  {\it continuous index function}
$C(x), x \in   \overline{Q}^{\delta},$ and study its properties.
\begin{equation}
C (x) = \max_{y \in \overline{Q}^{\delta}} \{ J(x),\hspace{1mm}
J(y)- K_{\nu(y)}\mid x- y \mid) \},          \label{6}
\end{equation}
where $K_{\nu(y)}$ such that $L_{\nu(y)} < K_{\nu(y)} < \infty$ is
an overestimate of the Lipschitz constant corresponding to the
function $g_{\nu(y) }(y)$ and $J(x)$ is the discontinuous index
function from (\ref{indexFun1}).

As an illustration, the function $C(x)$ corresponding to the
problem presented in Fig.~\ref{f1} is shown in Fig.~\ref{f2}. The
parts of the function  $C(x)$ corresponding to  $x$ and $y$ such
that
\[
J(x)< J(y)- K_{\nu(y)}\mid x- y \mid
\]
are shown by the thin line.
\begin{figure}[t]
  \begin{center}
    \caption{The function $C(x)$ corresponding to the
problem presented in Fig.~\ref{f1}.  \label{f2}}
    \epsfig{ figure = 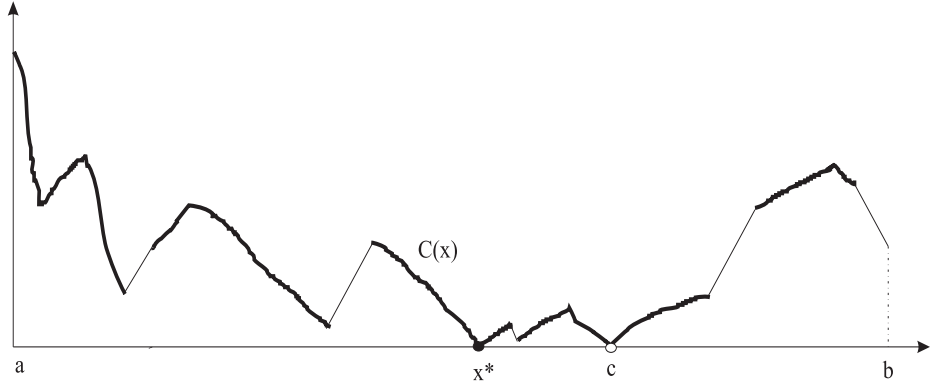, width = 4.8in, height = 2in,  silent = yes }
  \end{center}
\end{figure}

If $Q^{\delta}\neq \emptyset$, the global minimizers of the
original constrained problem (\ref{problem})--(\ref{const})
coincide with the so\-lu\-tions
 of the following  continuous problem
\begin{equation}
 C(x^{*})=\min \{ C(x): x\in \overline{Q}^{\delta} \}.  \label{C7}
  \end{equation}
In the case $Q^{\delta}= \emptyset$ the set
$\overline{Q}^{\delta}= [a,b]\setminus \ Q_{m+1} $ and we have
$\nu (x)<m+1, x \in \overline{Q}^{\delta}$. Thus, due to
Lemma~\ref{l1}, it follows
\begin{equation}
C (x)>0, \hspace{3mm} x \in \overline{Q}^{\delta}. \label{C71}
  \end{equation}

Similarly to definition of the function $J^k(x)$, the value
$Z^{*}_{k}$ is used to define the functions $C^k (x), x \in
[a,b],$ as follows.
\begin{equation}
C^k(x) = \max_{y \in [a,b]} \{ J^k(x),\hspace{1mm} J^k(y)-
K_{\nu(y)}\mid x- y \mid) \}.         \label{Ck}
\end{equation}

\begin{lemma}
 \label{l2}
The following assertions hold for the functions $C (x)$ and $
C^k(x)$:
\begin{enumerate}
\item[i.]
inequalities
 $C (x) \ge J (x),  \hspace{2mm} C^k(x) \ge J^k(x)$ hold over the set $\overline{Q}^{\delta}$;
\item[ii.]
  if \hspace{1mm} $\nu (x)<m+1$ then $C^k(x)>0$;
\item[iii.]
 if \hspace{1mm}$\nu (x)=m+1$ and $ x \in W(k)$ then $C^k(x)\ge 0$;
\item[iv.]
 if \hspace{1mm}$ x \in \{ x: g_{m+1}(x) \le
  Z^{*}_{k} \}$ then $  C^k (x) \le 0$.
\item[v.]
if \hspace{1mm} $x^{*}_{k} \in Q^{\delta}$ then $C^k (x) \le
C(x),\hspace{5mm} x \in  Q^{\delta}$.
\end{enumerate}
\end{lemma}

\begin{proof}
The truth of the assertions  follows from Lemma~\ref{l1} and
formulae (\ref{Zk}),(\ref{xk}), (\ref{6}), and (\ref{Ck}).
 \rule{5pt}{5pt}
 \end{proof}

It follows from Lemma~\ref{l2}  that  if $x^{*}_{k} \in
Q^{\delta}$, the global minimizers cannot be located in zones
where $C^k(x)>0,x \in Q^{\delta}$. Over every interval
$[l_{i},r_{i}]$ we are interested in subregions having the index
greater or equal to
\[
\overline{\nu_{i}}= \max \{ \nu(l_{i}),\nu(r_{i}) \},
\]
because, due to construction of the function  $C^k(x)$, only these
subregions can probably contain a global minimizer. It can be
shown  (see \cite{Sergeyev_Famularo_Pugliese}) that
\begin{equation}
[l_{i},r_{i}] \cap Q_{\overline{\nu_{i}}} \subseteq \left\{
\begin{array}{ll}
              $[$y^{-}_{i},y^{+}_{i}$]$, & \hspace{5mm} \nu(l_{i}) = \nu(r_{i}), \\
              $[$y^{-}_{i} , r_{i}$]$,     & \hspace{5mm} \nu(l_{i}) < \nu(r_{i}), \\
              $[$l_{i} , y^{+}_{i}$]$,     & \hspace{5mm} \nu(l_{i}) >
              \nu(r_{i}),
             \end{array}
       \right.               \label{intersection}
\end{equation}
where
 \beq y^{-}_{i}= l_{i}+z(l_{i})/K_{\nu(l_{i})},    \label{y-}
 \eeq
 \beq
  y^{+}_{i}= r_{i}-z(r_{i})/K_{\nu(r_{i})},          \label{y+}
\eeq
 and $z(x)=J^k(x)$.  Let us call any value $ R_i,1\le i\le k,$
{\it characteristic} of the interval $[l_{i},r_{i}]$ if the
following inequality is true
\begin{equation}
\min  \{ C^k (x): x \in [l_{i},r_{i}], \nu(x)=
\overline{\nu_{i}}\} \ge R_i. \label{minRi}
\end{equation}
 It follows from assertion i of Lemma~\ref{l2} and
\cite{Sergeyev_Famularo_Pugliese} that (\ref{minRi}) is fulfilled
for $ R_i = \check{R}_i$ where
\begin{equation}
\check{R}_i=\check{R}(l_{i},r_{i})=\left\{
\begin{array}{ll}
              0.5(z(l_{i})+z(r_{i})-K_{\nu(r_{i})}(r_{i}-l_{i})), &
\hspace{5mm} \nu(l_{i})=\nu(r_{i}),\\
              z(r_{i})-K_{\nu(r_{i})}(r_{i}-y^{-}_{i}), & \hspace{5mm} \nu(l_{i})<\nu(r_{i}), \\
              z(l_{i})-K_{\nu(l_{i})}(y^{+}_{i}-l_{i}), & \hspace{5mm}
              \nu(l_{i})>\nu(r_{i}).
             \end{array}
       \right.                  \label{Ri}
\end{equation}
The characteristic  $ \check{R}_i$ from (\ref{Ri}) depends only on
the values of the function $J^k(x)$ evaluated at the points
$l_{i}$ and $r_{i}$. It does not use any information from other
intervals belonging to the working list $W(k)$.

We are ready now to introduce the Algorithm working with
Continuous Index Functions (ACIF). It either  solves the problem
(\ref{C7}) or determines that the case (\ref{C71}) takes place.
The ACIF works by calculating characteristics $R_i$ initially
using (\ref{Ri}) and then improving them during the search by
constructing the function $C^k(x)$. On the one hand, the method
tries to find a good estimate $Z^{*}_{k}$. On the other hand, it
searches and eliminates from $W(k)$ intervals that cannot contain
$x^*$ using the fact  following from Lemma~\ref{l2} and
(\ref{minRi}) that if $x^{*}_{k} \in Q^{\delta}$, an interval
$[l_{j},r_{j}]$ having a characteristic $R_j>0$ can be eliminated
from consideration. The constraint introducing the parameter
$\delta$ helps to exclude more intervals.

Let us take a generic interval $[l_{t},r_{t}], 1 \le t \le
q(k+1),$ from the working list and calculate its characteristic
$R(l_{t},r_{t})$. We will also show how the function $C^k(x)$
allows us to improve characteristics of   intervals adjacent to
$[l_{t},r_{t}]$.

Initially  characteristic for the interval $[l_{t},r_{t}]$ is
calculated as $ R(l_{t},r_{t})=\check{R}(l_{t},r_{t})$. If
$R(l_{t},r_{t})\le 0$ or $\nu (l_{t})= \nu (r_{t})$, then the
characteristic $R(l_{t},r_{t})$ has been computed. If
$R(l_{t},r_{t})> 0$ and $\nu (l_{t}) < \nu (r_{t})$ go to the
operation  {\it Backward motion.} Otherwise execute the operation
{\it Onward motion.}

{  \it Backward motion.}  Exclude from $W(k+1)$ all the intervals
$i$ such that
 \beq
z(r_{t})-K_{\nu(r_{t})}(r_{t}-l_{i}))>0,\hspace{5mm} 1 \le j+1 \le
i \le t-1,    \label{step2}
 \eeq
where the  interval $j$ violates (\ref{step2}). Calculate the
value
\begin{equation}
R^{-}_j=\left\{
\begin{array}{ll}
              0.5(z(l_{j})+z^{-}(r_{j})-K_{\nu(r_{t})}(r_{j}-l_{j})), &
\hspace{5mm} \nu(l_{j})=\nu(r_{t})\\
              z^{-}(r_{j})-K_{\nu(r_{j})}(r_{j}-l_{j}-z(l_{j})/K_{\nu(l_{j})}), & \hspace{5mm} \nu(l_{j})<\nu(r_{t}) \\
              z(l_{j})-K_{\nu(l_{j})}(r_{j}-l_{j}-z^{-}(r_{j})/K_{\nu(r_{t})}), & \hspace{5mm} \nu(l_{j})>\nu(r_{t})
             \end{array}
       \right.                  \label{step21}
\end{equation}
where
 \beq
 z^{-}(r_{j}) = z(r_{t})-K_{\nu(r_{t})}(r_{t}-r_{j})).  \label{step22}
 \eeq
If $R^{-}_j<R_j$, set in the working list $W(k+1)$
 \[
z(r_{j})=z^{-}(r_{j}),\hspace{5mm}\nu(r_{j})=\nu(r_{t}),\hspace{5mm}
R_j=R^{-}_j,
\]
maintaining in the history list $H(k+1)$ the original values of
$g_{\nu (r_{j})}(r_{j})$ and $\nu(r_{j})$. Calculate the number
$q(k+1)$ of the intervals in $W(k+1)$.

An illustration to the operation  {\it Backward motion} is given
in Fig.~\ref{f3}. Three intervals are presented in Fig.~\ref{f3}:
\[
[l_{i-2},r_{i-2}]=[p,q], \hspace{5mm} [l_{i-1},r_{i-1}]=[q,h],
\hspace{5mm} [l_{i},r_{i}]=[h,d].
\]
Suppose that the function $C^k(x)$ has been evaluated at the
points $p,q,h,$ and $d$ and
 \[
\nu(p)=\nu(q)=\nu(h)<\nu(d).
\]
Characteristic $\check{R}_{i-2}$ of the interval $[p,q]$ is
negative and the bold line shows the zone where the global
minimizer could be probably found. The same situation holds for
the interval $[q,h]$. Since the characteristic $\check{R}_{i}$ of
the interval $[h,d]$ is positive and $\nu(h)<\nu(d)$, the
operation  {\it Backward motion} starts to work. It can be seen
from Fig.~\ref{f3} that the new characteristics $R_{i-2}$ and
$R_{i-1}$ calculated using information obtained at the point $d$
are positive and, therefore, the intervals $[p,q]$ and $[q,h]$
cannot contain   global minimizers and can be so excluded from the
working list.

\begin{figure}[t]
  \begin{center}
    \caption{Improving characteristics by the operation  {\it Backward motion}.  \label{f3}}
    \epsfig{ figure = 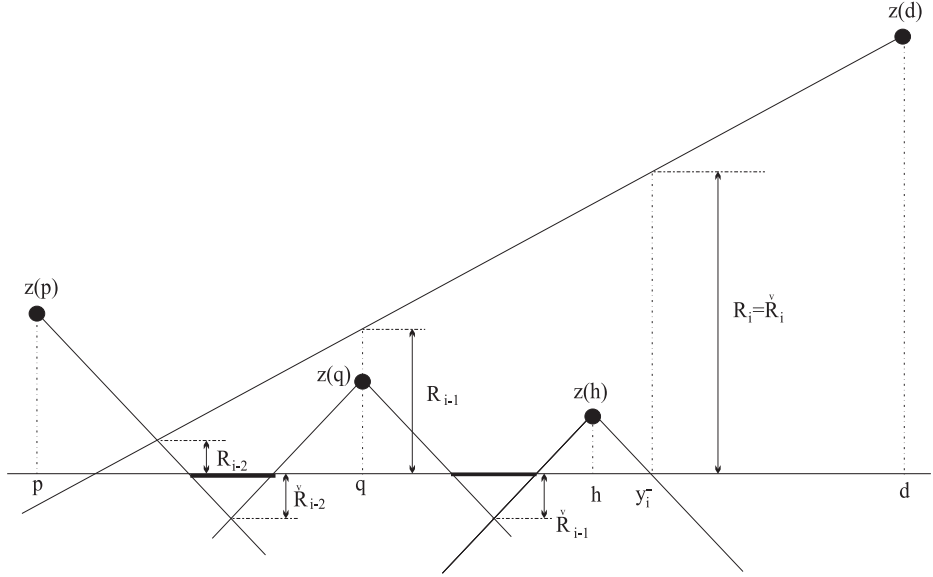, width = 4.8in, height = 3.in,  silent = yes }
  \end{center}
\end{figure}

{  \it Onward motion.} Exclude from $W(k+1)$ all the intervals
$i$ such that
 \beq
z(l_{t})-K_{\nu(l_{t})}(r_{i}-l_{t}))>0,  \hspace{5mm} t+1 \le i
\le j-1 \le q(k),     \label{step3}
 \eeq
where the  interval $j$ violates (\ref{step3}). Calculate the
value
\begin{equation}
R^{+}_j=\left\{
\begin{array}{ll}
              0.5(z^{+}(l_{j})+z(r_{j})-K_{\nu(r_{j})}(r_{j}-l_{j})), &
\hspace{5mm} \nu(l_{t})=\nu(r_{j})\\
              z(r_{j})-K_{\nu(r_{j})}(r_{j}-l_{j}-z^{+}(l_{j})/K_{\nu(l_{t})}), & \hspace{5mm} \nu(l_{t})<\nu(r_{j}) \\
              z^{+}(l_{j})-K_{\nu(l_{t})}(r_{j}-l_{j}-z(r_{j})/K_{\nu(r_{j})}), & \hspace{5mm} \nu(l_{t})>\nu(r_{j})
             \end{array}
       \right.                  \label{step31}
\end{equation}
where
 \beq
 z^{+}(l_{j}) = z(l_{t})-K_{\nu(l_{t})}(l_{j}-l_{t})).  \label{step32}
 \eeq
If $R^{+}_j<R_j$, set in the working list $W(k+1)$
 \[
z(l_{j})=z^{+}(l_{j}),\hspace{5mm}\nu(l_{j})=\nu(l_{t}),\hspace{5mm}
R_j=R^{+}_j,
\]
maintaining in the history list $H(k+1)$ the original values of
$g_{\nu (l_{j})}(l_{j})$ and $\nu(l_{j})$. Calculate the number
$q(k+1)$ of the intervals in $W(k+1)$.

In order to describe the method we need some definitions and
initial settings. It is supposed that:
 \bd
\item
 -- the  search accuracy $0 < \varepsilon \le \delta$  has been
chosen, where $\delta$ is from (\ref{problem});
\item
 -- two initial trials have been executed at the points $x^{0}=a$ and
$x^{1}=b$;
\item
 --   It has been assigned $W(1)=H(1)=[x^{0},x^{1}]$ and the number $t$
of the interval to be subdivided at the next iteration has been
set to $t=1$;
 \item
 -- the values $Z^{*}_{k}$ and
 \beq
 M^k= \max \{  \nu (x^{i}): 0\le i\le k \}          \label{Mk}
 \eeq
have been calculated for $k=1$;
 \item
 -- the set $V^\delta$ containing the points $x \in H(k)$ such
 that $x \in Q_{m+1}\setminus Q^{\delta}$ has been set to $V^\delta=
 \emptyset$.
 \ed

Suppose now that  $k, k\ge 1,$ iterations have been made by the
ACIF, the list $H(k)$ contains $k$ intervals, $W(k)$ contains
$q(k)$ intervals for which  characteristics have been evaluated,
and an interval $[l_{t},r_{t}]$ for subdivision has been found.
The choice of the next interval to be subdivided is made as
follows.

\begin{enumerate}
 \item[]
  {\hspace{5mm} \bf Step 1.  \it (Subdivision and the new
trial.)} Update $W(k+1)$ and $H(k+1)$ by substituting the interval
$[l_{t},r_{t}]$ in $W(k)$ and $H(k)$ by the new intervals
$[l_{t},x^{k+1}]$, $[x^{k+1},r_{t}]$ where
 \beq
x^{k+1}=\left\{ \begin{array}{ll}
 0.5(y^{-}_{t}+y^{+}_{t}), &  \hspace{5mm} \nu(l_{t})=\nu(r_{t})\\
0.5(y^{-}_{t}+r_{t}), &  \hspace{5mm} \nu(l_{t})<\nu(r_{t})\\
0.5(l_{t}+y^{+}_{t}), &  \hspace{5mm} \nu(l_{t})>\nu(r_{t})
             \end{array}
       \right.            \label{step1}
 \eeq
 Execute the $(k+1)$-th trial at the point $x^{k+1}$ and, as
the result, obtain the values $\nu (x^{k+1})$ and $g_{\nu
(x^{k+1})}(x^{k+1})$. Recalculate $M^{k+1}$.
\item[]
{\hspace{5mm} \bf Step 2. \it (Calculation of  the estimate
$Z^{*}_{k+1}$ and characteristics.) } Associate with the point
$x^{k+1}$ the value $z^{k+1}=J^{k+1} (x^{k+1})$ and recalculate
the estimate $Z^{*}_{k+1}$ if  $\nu (x^{k+1})=m+1$. If
$Z^{*}_{k+1}<Z^{*}_{k}$ then for all points $x \in W(k+1), x \neq
x^{k+1},$ such that $\nu(x)=m+1$  set
$z(x)=z(x)+Z^{*}_{k}-Z^{*}_{k+1}$ and recalculate characterisitcs
of the intervals in $W(k+1)$. Otherwise calculate characteristics
only  for the intervals $[l_{t},x^{k+1}]$ and $[x^{k+1},r_{t}]$.
Go to Step 3.
\item[]
{\hspace{5mm} \bf Step 3. \it (Finding an interval for the next
subdivision.)} If $W(k+1)=\emptyset$, then Stop (the feasible
region is empty). Otherwise, find in the working list $W(k+1)$ an
interval $[l_{t},r_{t}]$ such that
\begin{equation}
t = \min \{ \arg \min \{ R_i: 1\le i\le q(k+1) \} \}
\label{step4}
\end{equation}
and go to Step 4.
\item[]
{\hspace{5mm} {\bf Step 4.} \it (Verifying appurtenance to the
set $Q^\delta$.)} If the interval to be subdivided can belong to
the set $Q^\delta$ then go to Step~5. Otherwise exclude all found
intervals that are out of $Q^\delta$ from the working list and
include the points forming these intervals and having the index
$m+1$ in the set $V^\delta$. If the point $x^{*}_{k+1}$ belongs
to one of the excluded intervals   then go to Step~6 otherwise go
to Step~3.
\item[]
 {\hspace{5mm} \bf Step 5. \it (Verifying accuracy.)} If the inequality
 \begin{equation}
r_{t}-l_{t} > \varepsilon             \label{step6}
\end{equation}
holds, then go to Step 1. In the opposite case, Stop (the required
accuracy has been reached).
\item[]
{\hspace{5mm} \bf Step 6. \it (Restarting.)} Recalculate the
estimate $Z^{*}_{k+1}$ without usage of  the points included in
$V^\delta$. Form the new set $W(k+1)$ including in it all the
intervals from $H(k+1)$ that do not contain points from $V^\delta$
and intervals containing points $x \in V^\delta$ such that $z(x) >
Z^{*}_{k+1}$. For all intervals in $W(k+1)$ recalculate
characteristics $R_i$ applying backward motion for all intervals
$i$ having $R_i>0 $ if $\nu (l_{t})< \nu (r_{t})$ and onward
motion if $\nu (l_{t})> \nu (r_{t})$. In the latter case,
characteristics of the intervals satisfying (\ref{step3}) are not
calculated. Exclude from $W(k+1)$  all the intervals having
positive characteristics. Then go to Step~4.
 \end{enumerate}

Step  4 executes an important operation -- verifying appurtenance
to the set $Q^\delta$.  To do this we check whether  the  interval
$[l_{t},r_{t}]$ chosen for subdivision can contain a feasible
interval having a length greater than $\delta$. Four cases should
be considered.
 \begin{enumerate}
 \item[]
 { i. \it (Case $\nu (l_{t})< m+1$, $\nu (r_{t})< m+1$.)}  If
 \begin{equation}
y^{+}_{t}-y^{-}_{t}=r_{t}-l_{t}-z(r_{t})/K_{\nu(r_{t})}-z(l_{t})/K_{\nu(l_{t})}
< \delta  \label{step51}
\end{equation}
then $[l_{t},r_{t}] \notin Q^\delta$ because over $[l_{t},r_{t}]$
only the interval $[y^{-}_{t},y^{+}_{t}]$ can possibly contain a
global minimizer but its length is less than $\delta$.
\item[]
{ ii. \it (Case $\nu (l_{t})= m+1$, $\nu (r_{t})<m+1$.)}
Analogously, if
 \[
r_{t}-l_{t}-z(r_{t})/K_{\nu(r_{t})} > \delta
\]
then the interval $[l_{t},r_{t}]$ can belong to the set
$Q^\delta$. Otherwise, if in the history list $H(k+1)$ there
exists an interval $[l_{j},r_{j}], j<t,$  such that $\nu(l_{j})<
m+1$   and
 \begin{equation}
r_{t}-l_{j}-z(r_{t})/K_{\nu(r_{t})}-z(l_{j})/K_{\nu(l_{j})} <
\delta  \label{step521}
\end{equation}
or $\nu(l_{j})= m+1,j=1,$ and
 \begin{equation}
r_{t}-l_{j}-z(r_{t})/K_{\nu(r_{t})} < \delta  \label{step522}
\end{equation}
then all the intervals $[l_{j},r_{j}],\ldots,[l_{t},r_{t}] \notin
Q^\delta$ and the corresponding points $r_{j},\ldots,l_{t} \in
Q_{m+1}\backslash Q^\delta$.
\item[]
 { iii. \it (Case $\nu (l_{t})< m+1$, $\nu
(r_{t})=m+1$.)} This case is considered analogously to the
previous one but confirmation of  possibility for $[l_{t},r_{t}]$
to belong to $Q^\delta$ is searched among intervals $i>t$.
\item[]
 { iv. \it (Case $\nu (l_{t})=\nu
(r_{t})=m+1$.)} This case is a combination of the cases ii
and~iii.
 \end{enumerate}

The introduced procedure verifies inclusion  $[l_{t},r_{t}] \in
Q^\delta$ for all possible combinations of indexes $\nu
(l_{t}),\nu (r_{t})$. Of course, it is also possible to simplify
Step 4 and  verify only condition (\ref{step51}) -- the rule
determining during the search the major part of intervals
belonging to $Q_{m+1} \setminus Q^\delta$. In this case, after
satisfying  the stopping rule from Step 5, it is necessary to
check whether the found solution $x^{*}_{k+1}$ belongs to $Q_{m+1}
\setminus Q^\delta$ and, if necessary, to reiterate the method
starting from Step 6.

The following situations  can, therefore,  hold after fulfillment
of the stopping rule:
\begin{enumerate}
\item[i.]
The algorithm has finished its work and the working list is
empty, then $ Q^\delta = \emptyset$ and the set $ V^\delta$
contains the points from $Q_{m+1} \backslash Q^\delta $ if any.
\item[ii.]
The working list is not empty and it does not contain intervals
$[l_{p},r_{p}]$ such that $R_p<0$ and
 \beq
 \max \{\nu(l_{p}), \nu(r_{p}) \} < m+1.          \label{maxp}
 \eeq
In this case it is necessary to check locally in the neighborhood
of $x^{*}_{k}$ whether $x^{*}_{k} \in  Q^\delta $. If this
situation holds, then the global minimum $z^{*}$ of the problem
(\ref{problem})--(\ref{const}) can be bounded as follows
\[
z^{*} \in [R_{t(k)}+Z^{*}_{k}, Z^{*}_{k}]
\]
where  $R_{t(k)}$ is the characteristic corresponding to the
interval  number $t=t(k)$ from (\ref{step4}). In the opposite
case it is necessary to include the point $x^{*}_{k}$ in
$V^\delta$ and to return to Step~6.
\item[iii.]
The last case considers the situation where the working list is
not empty  and there exists an interval $[l_{p},r_{p}]$ such that
$R_p < 0$ and (\ref{maxp}) holds. Again, it is necessary to check
locally in the neighborhood of $x^{*}_{k}$ whether $x^{*}_{k} \in
Q^\delta $. If this analysis shows that $x^{*}_{k} \notin Q^\delta
$ then it is necessary to include the point $x^{*}_{k}$ in
$V^\delta$ and return to Step~6. Otherwise, the value $Z^{*}_{k}$
can be taken as an upper bound of the global minimum $z^{*}$. A
lower bound can be calculated easily by taking from the working
list the trial points $x_i$ such that $\nu(x_i)= m+1$ and
constructing for $f(x)$  the support function of the type
\cite{Pijavskii_(1972)} using only these points. The global
minimum of this support function over the intervals belonging to
the working list will be a lower bound for $z^{*}$.
\end{enumerate}

Consider now the infinite trial sequence $\{x^k\}$ generated by
the algorithm ACIF when $\varepsilon=0$ in the stopping rule
(\ref{step6}). We denote by $X^*$ the set of the global
minimizers of the problem (\ref{problem})--(\ref{const}) and by
$X'$ the set of limit points of the sequence $\{x^{k}\}$. The
following two theorems describe convergence conditions of the
ACIF. Since they can be  derived as a particular case of general
convergence studies given in \cite{Horst_and_Tuy_(1996)}
(Branch-and-Bound approach) and \cite{Sergeyev_(1999)} (Divide
the Best algorithms) their proofs are omitted.

\begin{theorem}
If the problem (\ref{problem})--(\ref{const}) is feasible, i.e. $
Q^\delta \neq \emptyset,$ then $X^*=X'$.
\end{theorem}

\begin{theorem} If the  problem (\ref{problem})--(\ref{const})
is infeasible then the algorithm ACIF stops in a finite number of
iterations.
\end{theorem}

\section{Numerical comparison and conclusion}

The ACIF has been numerically compared to the algorithm (indicated
hereinafter as {\it PEN}) proposed by Pijavskii (see
\cite{Pijavskii_(1972),Hansen_(3)}) combined with a penalty
function. The PEN has been chosen for comparison because the
method of Pijavskii in literature (see
\cite{Floudas_and_Pardalos_(1996),Horst_and_Pardalos_(1995),Horst_and_Tuy_(1996),Mladineo_(1992),Pardalos_and_Rosen_(1990),Pinter_(1996),Strongin_(1978)})
is used as a kind of the unit of measure of efficiency of the new
Lipschitz global optimization algorithms and it uses in its work
the same information about the problem as the ACIF -- the
Lipschitz constants for the objective function and constraints.
The usage of the penalty scheme allows us to emphasize advantages
of the index approach.

Since the PEN in every iteration evaluates the objective function
$f(x)$ and {\it all} the constraints, twenty feasible test
problems (ten differentiable and ten non-differentiable)
introduced in \cite{Famularo_Sergeyev_Pugliese} have been used for
testing the new algorithm. The ACIF has also been applied to one
differentiable and one non-differentiable infeasible test problems
from \cite{Famularo_Sergeyev_Pugliese}.   In all the experiments
there has been chosen the original  (see
\cite{Famularo_Sergeyev_Pugliese}) order the constraints are
evaluated during optimization, without determining the best for
the ACIF order.

In the PEN, the constrained problems were reduced to the
unconstrained ones as follows
 \beq P^*(x)=f(x) + P
\max\left\{g_1(x),g_2(x),\dots,g_{N_v}(x),0\right\}  \label{pen}
\eeq
and  coefficients $P$ from \cite{Famularo_Sergeyev_Pugliese} have
been used. The same accuracy $ \varepsilon = 10^{-4} \left(b -
a\right)$ (where $b$ and $a$ are from (\ref{problem})) and the
starting trial points $a$ and $b$ have been used in all the
experiments for both ACIF and PEN.

Table~\ref{t1} contains numerical results obtained for the PEN.
The column ``Evaluation" shows the total number of evaluations
equal to
\[  (N_v +1)\times N_{iter}, \]
where $N_v$ is the number of constraints and $N_{iter}$ is the
number of iterations for each problem.

\begin{center}
\begin{table}[t]\centering
\caption{Numerical results  obtained by the PEN
 on $10$ non-differentiable and $10$ differentiable problems.\label{t1}}
{\small
    \begin{tabular}{c|rr|rr} \hline\noalign{\smallskip}
 Problem    & \multicolumn{2}{c|}{Non-differentiable}  & \multicolumn{2}{c}{Differentiable}\\
 & Iterations & Evaluations& Iterations & Evaluations \\
\noalign{\smallskip}\hline\noalign{\smallskip}
 1 &    $247$ & $494$ &   $83$ & $166$ \\
 2 &    $241$ & $482$&   $954$ & $1906$   \\
 3 &    $797$ & $1594$&   $119$ & $238$   \\
 4 &    $272$ & $819$&    $1762$ & $5286$  \\
 5 &    $671$ & $2013$&   $765$ & $2295$  \\
 6 &    $909$ & $2727$&    $477$ & $1431$ \\
 7 &    $199$ & $597$ &   $917$ & $2751$  \\
 8 &    $365$ & $1460$&   $821$ & $3284$ \\
 9 &    $1183$ & $4732$&  $262$ & $1048$  \\
10 &    $135$ & $540$&   $2019$ & $8076$  \\
 Average & $501.9$ & $1545.8$ & $817.9$ & $2648.1$  \\ \noalign{\smallskip}\hline
    \end{tabular}
    }
\end{table}
    \end{center}

\begin{center}
\begin{table}[t]\centering
\caption{Results obtained by the new algorithm
 on the non-differentiable problems.\label{t2}}
 {\small
    \begin{tabular}{c|rrrrrr|rrrrrr} \hline\noalign{\smallskip}
 Problem    & \multicolumn{6}{c|}{$\delta=\varepsilon$} & \multicolumn{6}{c}{$\delta=10\varepsilon$}\\

 &  $N_{g_1}$ & $N_{g_2}$ & $N_{g_3}$ & $N_{f}$ & Iter. & Eval.&  $N_{g_1}$ & $N_{g_2}$ & $N_{g_3}$ & $N_{f}$ & Iter. & Eval. \\
\noalign{\smallskip}\hline\noalign{\smallskip}
 1 &  $23$ & $-$ & $-$ & $28$ &   $51$ & $79$   & $23$ & $-$ & $-$ & $28$ &   $51$ & $79$\\
 2 &  $18$  & $-$ & $-$ & $16$ & $34$ & $50$    &  $17$  & $-$ & $-$ & $16$ & $33$ & $49$ \\
 3 &  $95$ & $-$ & $-$ & $18$ & $113$ & $131$   &  $80$ & $-$ & $-$ & $18$ & $98$ & $116$\\
 4 &  $107$ & $14$ & $-$ & $84$ & $205$ & $387$ &  $82$ & $11$ & $-$ & $84$ & $177$ & $356$  \\
 5 &  $153$ & $88$ & $-$ & $24$ & $265$ & $401$ &  $114$ & $66$ & $-$ & $24$ & $204$ & $318$  \\
 6 &  $16$ & $16$ & $-$ & $597$ & $629$ & $1839$ &  $16$ & $15$ & $-$ & $597$ & $628$ & $1837$ \\
 7 &  $52$ & $18$ & $-$ & $39$ &  $109$ & $205$  &  $49$ & $14$ & $-$ & $39$ &  $102$ & $194$\\
 8 &  $28$ & $11$ & $3$ & $21$ &  $63$ & $143$   &  $28$ & $11$ & $3$ & $21$ &  $63$ & $143$ \\
 9 &  $8$ & $81$ & $49$ & $183$ & $321$ & $1049$ &  $8$ & $59$ & $32$ & $183$ & $282$ & $954$ \\
10 &  $32$ & $3$ & $17$ & $13$ &  $65$ & $141$   &  $30$ & $2$ & $17$ & $13$ &  $62$ & $137$ \\
 Average &  $53.2$ & $33.0$ & $23.0$ & $102.3$ & $185.5$ &
$442.5$ &  $44.7$ & $25.4$ & $17.3$ & $102.3$ & $170.0$ & $418.3$\\
\noalign{\smallskip}\hline
    \end{tabular}
    }
\end{table}
    \end{center}

\begin{center}
\begin{table}\centering
\caption{Results obtained by the new algorithm
 on the differentiable problems.\label{t3}}
 {\small
    \begin{tabular}{c|rrrrrr|rrrrrr} \hline\noalign{\smallskip}
 Problem    & \multicolumn{6}{c|}{$\delta=\varepsilon$} & \multicolumn{6}{c}{$\delta=10\varepsilon$}\\

 &  $N_{g_1}$ & $N_{g_2}$ & $N_{g_3}$ & $N_{f}$ & Iter. & Eval.&  $N_{g_1}$ & $N_{g_2}$ & $N_{g_3}$ & $N_{f}$ & Iter. & Eval. \\
\noalign{\smallskip}\hline\noalign{\smallskip}
 1 &  $10$ & $-$ & $-$ & $13$ &   $23$ & $36$    &  $10$ & $-$ & $-$ & $13$ &   $23$ & $36$\\
 2 &  $199$  & $-$ & $-$ & $21$ & $220$ & $241$  &  $155$  & $-$ & $-$ & $21$ & $176$ & $197$ \\
 3 &  $40$ & $-$ & $-$ & $22$ & $62$ & $84$      &  $38$ & $-$ & $-$ & $22$ & $60$ & $82$\\
 4 &  $480$ & $127$ & $-$&$189$ & $796$ & $1301$ &  $212$ & $73$ & $-$ & $189$ & $474$ & $925$  \\
 5 &  $8$ & $13$ & $-$ & $122$ & $143$ & $400$   &  $8$ & $13$ & $-$ & $122$ & $143$ & $400$   \\
 6 &  $14$ & $55$ & $-$ & $18$ & $87$ & $178$    &  $13$ & $34$ & $-$ & $18$ & $65$ & $135$ \\
 7 &  $36$ & $13$ & $-$ & $241$ &  $290$ & $785$ &  $35$ & $13$ & $-$ & $241$ &  $289$ & $784$\\
 8 &  $94$ & $21$ & $5$ & $82$ &  $202$ & $479$  &  $80$ & $19$ & $5$ & $82$ &  $186$ & $461$ \\
 9 &  $7$ & $35$ & $6$ & $51$ & $99$ & $299$     &  $7$ & $32$ & $6$ & $51$ & $96$ & $293$ \\
10 &  $36$ & $14$ & $174$ & $1173$ &  $1397$ & $5278$  &  $35$ & $10$ & $92$ & $1173$ &  $1310$ & $5023$ \\
 Average &  $92.4$ & $39.7$ & $61.7$ & $193.2$ & $331.9$ &
$908.1$ &  $59.3$ & $27.7$ & $34.3$ & $193.2$ & $282.2$ & $833.6$\\
\noalign{\smallskip}\hline
    \end{tabular}
    }
\end{table}
    \end{center}

\vspace*{-25mm} Tables~\ref{t2} and~\ref{t3} present  numerical
results for the new method for $\delta=\varepsilon$ and
$\delta=10\varepsilon$. The columns in the Tables have the
following meaning:
\begin{enumerate}
\item [-] the columns $N_{g_1}$, $N_{g_2}$, and $N_{g_3}$   present
 the number of trials where the constraint $g_i, 1 \le i \le 3,$ was the last
 evaluated constraint;
\item [-] the column $N_{f}$ shows how many times the objective function $f(x)$
has been evaluated;
\item [-] the column "Eval." is the total number of evaluations
 of the objective function and the constraints. This quantity is equal
to:
\begin{enumerate}
\item [-]  $N_{g_1} + 2 \times N_{f} ,$\,for problems with one
constraint;
\item [-] $ N_{g_1} + 2 \times N_{g_2} + 3 \times N_{f},$\,  for problems with two
constraints;
\item [-] $ N_{g_1} + 2 \times N_{g_2} + 3 \times N_{g_3} + 4 \times N_{f},$\, for  problems with three
constraints.
\end{enumerate}
\end{enumerate}

It can be seen from the Tables that in all the experiments the
ACIF significantly outperforms the PEN both in iterations and
evaluations. The ACIF works faster if the difference between
$\delta$ and $ \varepsilon$ increases. This effect is especially
notable for problems where it is necessary to execute many
iterations out of the feasible region (see columns $N_{g_1}$,
$N_{g_2}$,  $N_{g_3}$ for non-differentiable problems 3--5, 9 and
differentiable problems 2, 4, 8, 10).

Note that the  penalty approach  requires an accurate tuning of
the penalty coefficient in contrast to the ACIF that works
without necessity to determine any additional parameter. Moreover,
when the penalty approach is used and a constraint $g(x)$ is
defined only over a subregion $[c,d]$ of the search region
$[a,b]$, the problem of extending $g(x)$ to the whole region
$[a,b]$ arises. The ACIF does not have this difficulty because
the constraints and the objective function are evaluated only
within their regions of definition.

Finally, the penalty approach is not able to determine whether a
problem is infeasible. The ACIF with $\delta=\varepsilon$ has
determined infeasibility of the non-differentiable problem   from
\cite{Famularo_Sergeyev_Pugliese} in $86$ iterations consisting of
$81$ evaluations of the first constraint and $5$ evaluations of
the first and second constraints (i.e., $91$ evaluations in
total). The infeasibility of the differentiable problem  from
\cite{Famularo_Sergeyev_Pugliese} has been determined by the ACIF
with $\delta=\varepsilon$ in $38$ iterations consisting of $9$
evaluations of the first constraint and $29$ evaluations of the
first and second constraints (i.e., $67$ evaluations in total).
Naturally, the objective functions have  not been evaluated in
both cases.

In conclusion, we illustrate performance of the new method (see
Fig.~\ref{f4}) and the PEN (see Fig.~\ref{f5}) on the
non-differentiable problem $9$ from
\cite{Famularo_Sergeyev_Pugliese}.
\[
\ds \min_{x \in \left[0 , 4 \right]}   f(x)  =   \ds 3 - 2
\exp\left(-\frac{1}{2} \left(\frac{22}{5}-x\right)\right)
\left|\sin\left( \pi \left(\frac{22}{5}-x\right) \right)\right|
\]
subject to
\[
\begin{array}{cccl}
 & g_1(x) & = & \ds 3 \left(
\exp\left(-\left|\sin\left(\frac{5}{2} \sin\left(\frac{11}{5} x
\right)\right)\right|\right)+\frac{1}{100} x^2 -\frac{1}{2}
\right) \le 0, \\[12pt] & g_2(x) & = & \ds \left\{
\begin{array}{cl} \ds 6 \left(x-\frac{1}{2}\right)^2-\frac{1}{2} &
x \le \frac{1}{2} \\[12pt]
                                          \ds \frac{1}{4} \left( x - \frac{5}{2} \right) & x > \frac{1}{2}
                    \end{array} \right. \le 0, \\[30pt]
& g_3(x) & = & \ds \frac{4}{5} -\left( \left| \sin \left(
\frac{24}{5} - x \right) \right| + \frac{6}{25} - \frac{x}{20}
\right) \le 0.
\end{array}
\]
The problem has $3$ disjoint feasible subregions shown in
Fig.~\ref{f4} by continuous bold intervals on the line $f(x)=0$,
the global optimum is located at the point $x^*=0.95019236$ (see
Fig.~\ref{f4}). The objective function is shown by a solid line
and the constraints are  drawn by dotted/mix-dotted lines.

\begin{figure}[t]
  \begin{center}
    \caption{Behaviour of the new method  on the non-differentiable problem 9 from
\cite{Famularo_Sergeyev_Pugliese}.  \label{f4}}
    \epsfig{ figure = 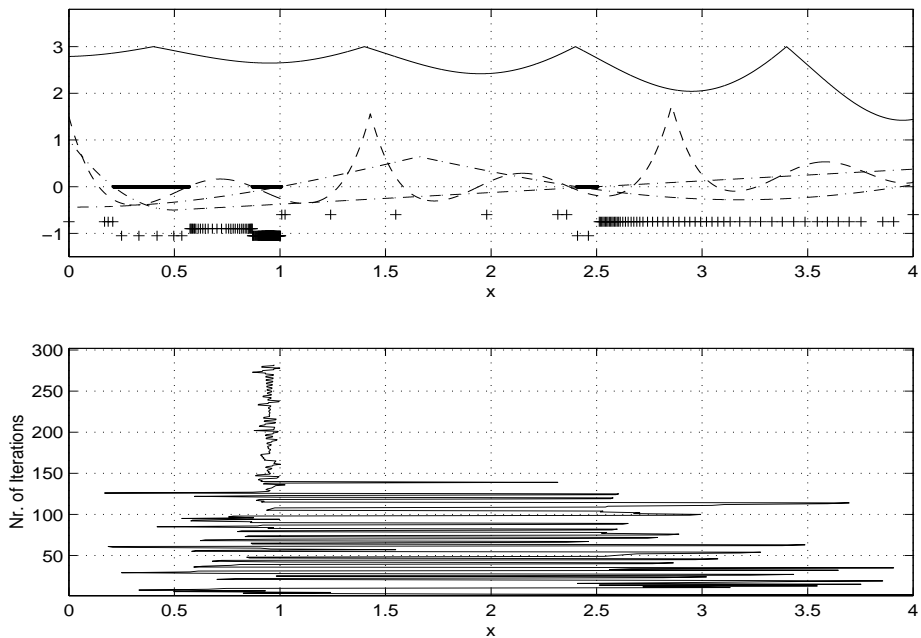, width = 4.8in, height = 3.375in,  silent = yes }
  \end{center}
\end{figure}

\begin{figure}[t]
  \begin{center}
    \caption{Behaviour of the PEN  on the non-differentiable problem $9$ from
\cite{Famularo_Sergeyev_Pugliese}.  \label{f5}}
    \epsfig{ figure = 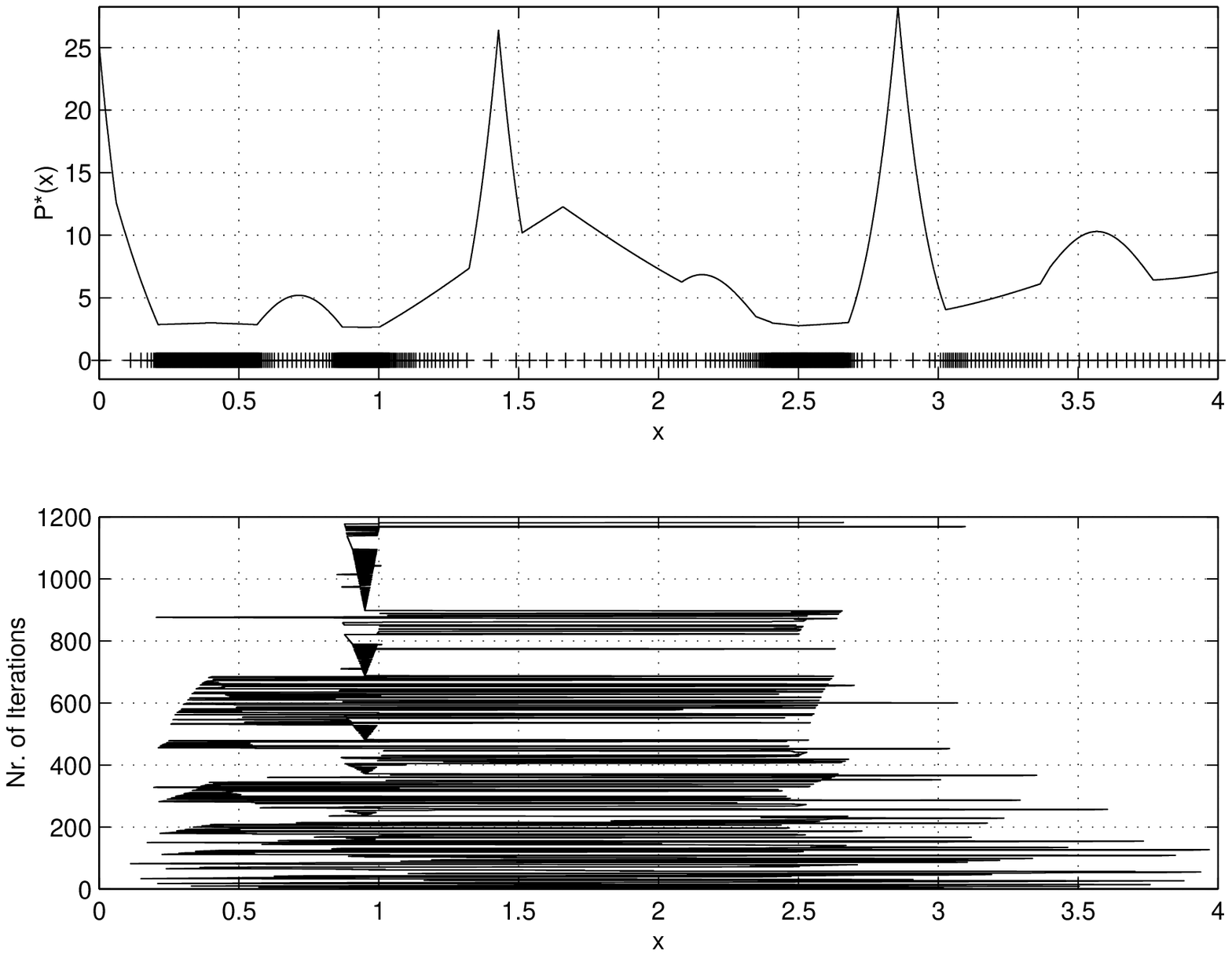, width = 4.8in, height = 3.375in,  silent = yes }
  \end{center}
\end{figure}

The first line (from up to down) of ``+'' located under the graph
of the problem 9 in the upper subplot of Fig.~\ref{f4} represents
the points where the first constraint has not been satisfied
(number of iterations equal to 8). Thus, due to the decision
rules of the ACIF, the second constraint has not been evaluated at
these points. The second line of ``+'' represents the points
where the first constraint has been satisfied but the second
constraint has been not (number of iterations equal to 59). In
these points both constraints have been evaluated but the
objective function has been not. The third line of ``+''
represents the points where both the first and the second
constraints have been satisfied but the third constraint has been
not (number of iterations equal to 32). The last line represents
the points where all the constraints have been satisfied and,
therefore, the objective function has been evaluated (number of
evaluations equal to 183). The total number of evaluations is
equal to $8+59 \times 2+32 \times 3 +183 \times 4 = 954$. These
evaluations have been executed during $8+59+32+183=282$
iterations. The lower subplot in Fig.~\ref{f4} shows dynamics of
the search.


 Fig.~\ref{f5} shows the penalty function
corresponding to $P=15$ and dynamics of the search executed by
the PEN. The line of ``+'' located under the graph in the upper
subplot of Fig.~\ref{f5} represents the points where the function
(\ref{pen}) has been evaluated. The number of iterations is equal
to $1183$ and the number of evaluations is equal to $1183 \times
4 = 4732$.

{}


\begin{thebibliography}{99}

\bibitem{Archetti_and_Schoen_(1984)} Archetti F. and F. Schoen (1984),
A survey on the global optimization problems: general theory and
computational approaches, {\it Annals of Operations Research},
{\bf 1},  87--110.

\bibitem{Bertsekas_(1996)}Bertsekas D.P. (1996), {\it Constrained Optimization and Lagrange Multiplier Methods}, Athena Scientific, Belmont, MA.

\bibitem{Bertsekas_(1999)}Bertsekas D.P. (1999), {\it Nonlinear Programming}, Second Edition, Athena Scientific, Belmont, MA.

\bibitem{Bomze97}Bomze I.M., T. Csendes, R.  Horst,  and  P.M. Pardalos (1997)  {\it Developments  in  Global  Optimization}, Kluwer  Academic Publishers, Dordrecht.

\bibitem{Breiman_and_Cutler_(1993)}Breiman L. and A. Cutler (1993), A deterministic algorithm for global optimization,{\it Math. Programming},  {\bf 58},  179--199.

\bibitem{Brooks_(1958)} Brooks S.H. (1958), Discussion of random methods for locating surface maxima, {\it Operation Research}, {\bf 6},  244--251.

\bibitem{Calvin_and_Zilinskas_(1999)} Calvin J. and A. \v{Z}ilinskas (1999), On the convergence of the P-algorithm for one-dimensional global optimization of smooth functions, {\it JOTA}, {\bf 102},  479--495.

\bibitem{Casado_GS(2000)}
  L. G. Casado, I. Garc\'{\i}a, and Y. D. Sergeyev (2000),
Interval branch
  and bound algorithm for finding the First-Zero-Crossing-Point in
  one-dimensional functions, {\it Reliable Computing}, {\bf 2},   179--191.

\bibitem{Casado_GS(2002)}
  L. G. Casado, I. Garc\'{\i}a, and Y. D. Sergeyev (2002),   Interval
algorithms for finding the minimal root in a set of multiextremal
non-differentiable one-dimensional functions, {\it SIAM J. on
Scientific Computing}, {\bf 24(2)}, 359--376.

\bibitem{DaponteGMS95}
 P. Daponte, D. Grimaldi, A. Molinaro, and Y. D. Sergeyev
(1995), {\it An
  algorithm for finding the zero crossing of time signals with Lipschitzean
  derivatives}, Measurements, {\bf 16},   37--49.

\bibitem{DaponteGMS96}
P. Daponte, D. Grimaldi, A. Molinaro, and Y. D. Sergeyev
  (1996),
Fast detection of
  the first zero-crossing in a measurement signal set,  {\it Measurements}, {\bf 19},   29--39.

\bibitem{Evtushenko_(1992)}Evtushenko Yu.G., M.A. Potapov and V.V. Korotkich (1992), Numerical methods for global optimization, {\it Recent Advances in Global Optimization}, ed. by C.A. Floudas and P.M. Pardalos, Princeton University Press, Princeton.

\bibitem{Famularo_Sergeyev_Pugliese}Famularo D., Sergeyev Ya.D.,   and P. Pugliese (2002), Test Problems  for Lipschitz Univariate Global Optimization with
Multiextremal Constraints, {\it Stochastic and Global
Optimization}, eds. G. Dzemyda, V. Saltenis and A. \v{Z}ilinskas,
Kluwer Academic Publishers, Dordrecht, 93-110.

\bibitem{Floudas_and_Pardalos_(1996)} Floudas C.A. and P.M. Pardalos (1996), {\it State of the Art in Global Optimization}, Kluwer Academic Publishers, Dordrecht.

\bibitem{Hansen_(3)} Hansen P. and B. Jaumard (1995), Lipshitz optimization. In:  Horst, R., and Pardalos, P.M. (Eds.). {\it Handbook of Global Optimization}, 407-493, Kluwer Academic Publishers, Dordrecht.

\bibitem{Horst_and_Pardalos_(1995)}Horst R. and P.M. Pardalos (1995), {\it Handbook of Global Optimization}, Kluwer Academic Publishers, Dordrecht.

\bibitem{Horst_and_Tuy_(1996)}Horst R. and H. Tuy (1996), {\it Global Optimization - Deterministic Approaches}, Springer--Verlag, Berlin, Third edition.

\bibitem{Lamar_(1999)}Lamar B.W. (1999), A method for converting a class of univariate functions into d.c. functions, {\it J. of Global Optimization}, {\bf 15}, 55--71.

\bibitem{Locatelli_and_Schoen_(1995)}Locatelli M. and F. Schoen (1995), An adaptive stochastic global optimisation algorithm for one-dimensional functions, {\it Annals of Operations research}, {\bf  58}, 263--278.

\bibitem{Locatelli_and_Schoen_(1999)}Locatelli M. and F. Schoen (1999), Random Linkage: a family of acceptance/rejection algorithms for global optimisation, {\it Math. Programming}, {\bf  85}, 379--396.

\bibitem{Lucidi_(1994)}Lucidi S. (1994), On the role of continuously differentiable exact penalty functions in constrained global optimization, {\it J. of Global Optimization}, {\bf 5},  49--68.

\bibitem{MacLagan_and_Sturge_and_Baritompa_(1996)}MacLagan D.,   Sturge, T., and  W.P. Baritompa (1996), Equivalent Methods for Global Optimization, {\it State of the Art in Global Optimization}, eds. C.A. Floudas, P.M. Pardalos, Kluwer Academic Publishers, Dordrecht,  201--212.

\bibitem{Mladineo_(1992)}Mladineo
R. (1992), Convergence rates of a  global optimization algorithm,
{\it Math. Programming}, {\bf 54},  223--232.

\bibitem{Molinaro_(2001)} Molinaro A., Sergeyev Ya.D. (2001) An efficient algorithm for
the zero-crossing detection in digitized measurement signal, {\it
Measurement}, {\bf 30(3)}, 187--196.

\bibitem{Nocedal_and_Wright_(1999)}Nocedal J. and S.J. Wright (1999),  {\it Numerical Optimization} (Springer Series in Operations Research), Springer Verlag.

\bibitem{Pardalos_and_Rosen_(1990)}Pardalos P.M. and J.B. Rosen (1990), Eds., {\it Computational Methods in Global Optimization, Annals of Operations Research}, {\bf 25}.

\bibitem{Patwardhan_(1987)}Patwardhan A.A., M.N. Karim and R. Shah (1987),Controller tuning by a least-squares method, {\it AIChE J.}, {\bf 33},  1735--1737.

\bibitem{Pijavskii_(1972)}Pijavskii S.A. (1972),
An Algorithm for Finding the Absolute Extremum of a Function, {\it
USSR Comput. Math. and Math. Physics}, {\bf 12}, 57--67.

\bibitem{Pinter_(1996)}Pint\'{e}r J.D. (1996), {\it Global Optimization in Action}, Kluwer Academic Publisher, Dordrecht.

\bibitem{Ralston_(1985)} Ralston P.A.S., K.R. Watson, A.A. Patwardhan and P.B. Deshpande (1985), A computer algorithm for optimized control, {\it Industrial and Engineering Chemistry, Product Research and Development}, {\bf 24},  1132.

\bibitem{Sergeyev_(1998)}Sergeyev Ya.D. (1998), Global one-dimensional optimization using smooth auxiliary functions, {\it Mathematical Programming}, {\bf     81}, 127-146.

\bibitem{Sergeyev_(1999)}Sergeyev Ya.D. (1999), On convergence of "Divide the Best" global optimization algorithms,   {\it Optimization}, {\bf 44}, 303--325.


\bibitem{Sergeyev_et(1999)}Sergeyev Ya.D., P. Daponte, D. Grimaldi and A. Molinaro (1999), Two methods for solving optimization problems arising in electronic
     measurements and electrical engineering,  {\it SIAM J. Optimization}, {\bf 10},  1--21.

\bibitem{Sergeyev_(2000)}Sergeyev Ya.D. (2000),   An Efficient Strategy for Adaptive
Partition of N-Dimensional Intervals in the Framework of Diagonal
Algorithms, {\it Journal of Optimization Theory and Applications},
{\bf 107},   145--168.

\bibitem{Sergeyev_Famularo_Pugliese}Sergeyev Ya.D., Famularo D.,    and P. Pugliese
(2001), Index Branch-and-Bound Algorithm for Lipschitz Univariate
Global Optimization with  Multiextremal Constraints, {\it J. of
Global Optimization}, {\bf 21}, 317--341.

\bibitem{Sergeyev_and_Markin_(1995)}Sergeyev Ya.D. and D.L. Markin (1995), An algorithm for solving global optimization problems with nonlinear constraints, {\it J. of Global Optimization}, {\bf 7},  407--419.


\bibitem{Strongin_(1978)}Strongin R.G. (1978), {\it Numerical Methods on Multiextremal Problems}, Nauka, Moscow, (In Russian).

\bibitem{Strongin_(1984)}Strongin, R.G. (1984), Numerical methods for multiextremal nonlinear programming problems with nonconvex constraints. In: Demyanov, V.F., and Pallaschke, D. (Eds.) {\it Lecture Notes in Economics and Mathematical Systems} 255, 278-282. Proceedings 1984.  Springer-Verlag. IIASA, Laxenburg/Austria.

\bibitem{Strongin_and_Markin_(1986)}Strongin R.G. and D.L. Markin (1986), Minimization of multiextremal functions with nonconvex constraints, {\it Cybernetics}, {\bf 22},  486--493.

\bibitem{Strongin_and_Sergeyev}Strongin R.G. and Ya.D.
Sergeyev (2000), {\it Global Optimization with Non-Convex
Constraints: Sequential and Parallel Algorithms}, Kluwer Academic
Publishers, Dordrecht.

\bibitem{Sun_and_Li_(1999)}Sun X.L. and D. Li (1999), Value-estimation function method for constrained global optimization, {\it JOTA}, {\bf 102},  385--409.

\bibitem{Torn_and_Zilinskas_(1989)}T\"orn A. and A. \v{Z}ilinskas (1989), {\it Global Optimization}, Springer--Verlag, Lecture Notes in Computer Science, {\bf 350}.

\bibitem{Wang_and_Chang_(1996)}Wang X. and T.S.Chang (1996), An improved univariate global optimization algorithm with improved linear bounding functions, {\it J. of Global Optimization}, {\bf 8},  393--411.

\bibitem{Zhigljavsky_(1991)}Zhigljavsky A.A. (1991), {\it Theory of Global Random Search}, Kluwer Academic Publishers, Dordrecht.

\end{thebibliography}
\end{document}